\documentclass[11pt]{amsart}
\usepackage{graphicx} 
\usepackage{amsmath, latexsym, amssymb, url, hyperref}
\usepackage{amsfonts,mathrsfs}
\usepackage{cleveref}
\numberwithin{equation}{section}
\usepackage[all]{xy}

\usepackage[centering, marginparwidth=5cm]{geometry}

\newcommand{\nc}{\newcommand}
\newcommand{\rnc}{\renewcommand}
\nc{\R}{\mathbb R}
\nc{\C}{\mathbb C}
\nc{\N}{\mathbb N}
\nc{\Z}{\mathbb Z}
\nc{\Q}{\mathbb Q}
\rnc{\P}{\mathbb P}
\nc{\F}{\mathbb F}
\nc{\Frac}{\mathrm{Frac}}
\rnc{\O}{\mathcal O}
\nc{\Tr}{\mathrm{Tr}}
\nc{\rmU}{\mathrm{U}}
\nc{\an}{\textnormal{an}}

\newcommand\cA{{\mathcal A}}
\newcommand\cB{{\mathcal B}}
\newcommand\cC{{\mathcal C}}
\newcommand\cD{{\mathcal D}}

\newcommand\cG{{\mathcal G}}

\newcommand\cO{{\mathcal O}}

\newcommand\cS{{\mathcal S}}

\newcommand\cV{{\mathcal V}}

\newcommand\cX{{\mathcal X}}
\newcommand\cY{{\mathcal Y}}
\newcommand\cZ{{\mathcal Z}}

\newcommand{\Aut}{\textnormal{Aut}}

\newcommand{\td}{\textnormal{td}}

\newcommand\cGamma{{\mathit \Gamma}}

\newcommand\bV{{\mathbb V}}

\newcommand\bW{{\mathbb W}}

\newcommand\bfG{{\mathbf G}}
\newcommand\bfH{{\mathbf H}}

\newcommand\bfN{{\mathbf N}}
\newcommand{\GL}{\mathbf{GL}}

\newcommand{\Hom}{\mathrm{Hom}}

\newcommand{\codim}{\mathrm{codim}}
\newcommand{\typ}{\mathrm{typ}}
\newcommand{\atyp}{\mathrm{atyp}}
\newcommand{\maxatyp}{\mathrm{max.atyp}}
\newcommand{\pos}{\mathrm{pos}}
\newcommand{\zero}{\mathrm{zero}}
\newcommand{\HL}{\mathrm{HL}}
\newcommand{\oQ}{\overline{\Q}}
\newcommand{\lo}{\rightarrow}
\newcommand{\ad}{\mathrm{ad}}
\newcommand{\dom}{\mathrm{dom}}
\newcommand{\Gr}{\textnormal{Gr}}
\newcommand{\Ad}{\textnormal{Ad}}

\renewcommand{\P}{\mathbb P}

\theoremstyle{plain}

\newtheorem{theorem}{Theorem}[section]

\newtheorem{proposition}[theorem]{Proposition}

\newtheorem{lemma}[theorem]{Lemma}

\newtheorem{corollary}[theorem]{Corollary}
\newtheorem{conjecture}[theorem]{Conjecture}

\theoremstyle{definition}
\newtheorem{definition}[theorem]{Definition}

\theoremstyle{remark}
\newtheorem{remarque}[theorem]{Remark}

\usepackage[usenames,dvipsnames]{color}

\usepackage{cleveref}

\title[Zilber-Pink-conjecture]{The Zilber-Pink conjecture for varieties not defined over $\overline{\Q}$}
\author{Bruno Klingler}
\address{Humboldt Universität zu Berlin, Department of mathematics, Rudower Chaussée 25, Room 1.403}
\email{bruno.klingler@hu-berlin.de}

\author{Salim Tayou}
\address{Department of Mathematics, Dartmouth College, Kemeny Hall, Hanover, NH 03755, USA}
\email{salim.tayou@dartmouth.edu}

\date{\today}

\begin{document}

\begin{abstract}

    In this note, we prove the Zilber--Pink conjecture for  subvarieties of mixed Shimura varieties, which are not defined over~$\oQ$ in a strong sense. We prove similar results for general variations of mixed Hodge structure of geometric origin, assuming furthermore that they are absolute.
\end{abstract}

\maketitle
\setcounter{tocdepth}{1}
\tableofcontents

\section{Introduction} \label{introduction}

\subsection{The Zilber-Pink conjecture for variations of mixed Hodge structures}
Let $X$ be a smooth quasi-projective algebraic variety over $\C$ and let $\bV=\{\bV_\Z,W_\bullet,F^\bullet\}$ be a variation of mixed Hodge structures \footnote{In this paper, all $\Z$VMHS are assumed to be admissible graded-polarizable} ($\Z$VMHS) over $X$. The Hodge locus $\HL(X, \bV^\otimes)$ of the variation $\bV$ is the subset of points~$x \in X(\C)$ where the Hodge structure $\bV_x$ admits "more Hodge tensors" than the very general fiber of $\bV$. 
When $\bV$ is of the form $\bV = R^ig_* \Z$ for some flat surjective equisingular morphism $g: Z \rightarrow X$, the Hodge conjecture predicts that the locus $\HL(X, \bV^\otimes)$ is a countable union of irreducible algebraic subvarieties of $X$: {\em the special subvarieties of $X$ for $\bV$}. A famous result of Cattani-Deligne-Kaplan \cite{cattani-deligne-kaplan} when $\bV$ is pure (see also \cite{bakker-klingler-tsimerman}), generalized to the mixed case in \cite{brosnan-pearlstein}, \cite{brosnan-pearlstein-schnell}, \cite{BBKT}, states that this holds true unconditionally for a general $\Z$VMHS $\bV$. 

Let $(\bfG,\cD)$ be the generic mixed Hodge datum of $\bV$ (thus $\bfG$ is the generic Mumford-Tate group of $\bV$ and $\cD$ the associated mixed Mumford-Tate domain, see \cite{klingler} for instance). The $\Z$VMHS $\bV$ is completely described by the complex analytic period map
\[\varphi:X^\an \rightarrow \Gamma\backslash \cD~,\] where $\Gamma\subset \bfG(\Q)$ is an arithmetic subgroup and $X^\an$ is the complex manifold analytification of $X$. In general, the complex locally homogeneous space $\Gamma\backslash \cD$ does not admit any algebraic structure. The special subvarieties $Z \subset X$ for $\bV$ are precisely the analytic irreducible components of $\varphi^{-1} (\Gamma_Z \backslash \cD_Z)$, where $(\bfG_Z, \cD_Z) \subsetneq (\bfG, \cD)$ is the generic Hodge datum of $\bV_{|Z}$, see \cite{klingler}, \cite{klingler-otwinowska}.  The content of the Cattani-Deligne-Kaplan theorem is that, although $\Gamma\backslash \cD$ and the period map $\varphi$ are not algebraic objects in general, these analytic irreducible components of $\varphi^{-1} (\Gamma_Z \backslash \cD_Z)$ are algebraic subvarieties of $S$. We will also write $\HL(X, \varphi)$ for the Hodge locus $\HL(X,\bV^\otimes)$.

\begin{remarque}
    When $\Gamma\backslash \cD$ admits an algebraic structure, the $\Z$VMHS $\bV$ is said of mixed Shimura type. In that case $\Gamma \backslash \cD$ is the analytification of a (connected) mixed Shimura variety $\cS$ and the period map $\varphi: X \lo \cS$ is algebraic.
\end{remarque}

The distribution of the special subvarieties of $X$ for $\bV$ has been the focus of much work in the last ten years. Using the period map associated to $\bV$, Klingler \cite[Def.4.3]{klingler}, generalized by Baldi-Klingler-Ullmo \cite[Def.5.1 and 5.2]{baldi-klingler-ullmo}, introduced a fundamental dichotomy: a special subvariety is said {\em typical} if its image under the period map has the dimension predicted by intersection theory:
\[ \codim_{\Phi(X^\an)} \, \Phi(Z^{\an}) = \codim_{\Gamma \backslash \cD} \,\Gamma_{Z} \backslash \cD_Z,\]
and is said {\em atypical} otherwise. This gives rise to a decomposition:
\[\HL(X, \bV^\otimes)=\HL(X, \bV^\otimes)_{\typ}\cup \HL(X, \bV^\otimes)_{\atyp},\]
where the typical Hodge locus $\HL(X, \bV^\otimes)_{\mathrm{typ}}$ is the countable union of the typical special subvarieties of $X$ for $\bV$ and the atypical Hodge locus $\HL(X, \bV^\otimes)_{\mathrm{atyp}}$ is the countable union of the atypical ones. In \cite[1.5]{klingler}, generalized by \cite[Conjecture 2.5]{baldi-klingler-ullmo}, the first author proposed the following conjecture, which extends to any $\Z$VMHS the classical Zilber-Pink conjecture for mixed Shimura varieties.

\begin{conjecture}[Zilber--Pink for $\Z$VMHS] \label{ZP}
For any smooth irreducible quasi-projective algebraic variety $X$ over $\C$ endowed with an admissible graded-polarizable variation of mixed Hodge structures $\bV$, the atypical Hodge locus $\HL(X, \bV^\otimes)_{\mathrm{atyp}}$ is not Zariski-dense in $X$. 
Equivalently: $\HL(X, \bV^\otimes)_{\mathrm{atyp}}$ is a union of finitely many maximal atypical special subvarieties of $X$ for $\bV$. 
\end{conjecture}

\begin{remarque}
    The definition of atypical subvarieties in \cite{baldi-klingler-ullmo} is more general than the one in \cite{klingler}, see \cite[Section 5.3]{baldi-klingler-ullmo} for the comparison. The equivalence mentioned in \Cref{ZP} is proven in \cite[Prop. 5.1]{klingler} using the restricted definition of atypicality. The proof extends unchanged to the more general definition.
\end{remarque}

When $\bV$ is of mixed Shimura type, the algebraic period map \[ \varphi: X \lo \cS\] decomposes as $\varphi = \iota \circ f$, where $Y \stackrel{\iota}{\hookrightarrow} \cS$ denotes the Zariski closure of the constructible set $\varphi(X)$ in the mixed Shimura variety $\cS$ and $f: X \lo Y$ the natural dominant projection. Although $Y$ is usually singular, we still call $\iota: Y \hookrightarrow \cS$ a period map and define $\HL(Y, \iota)$ and $\HL(Y, \iota)_\atyp$ as before. In particular $\HL(X, \varphi)_\atyp = f^{-1}(\HL(Y, \iota)_\atyp)$. Hence \Cref{ZP} for $\Z$VMHS of Shimura type reduces to the classical Zilber-Pink conjecture for mixed Shimura varieties \cite{Zilber}, \cite{Pink}.

\begin{conjecture} (Zilber-Pink)\label{ZP-shimura}
Let $X \stackrel{\iota}{\hookrightarrow} \cS$ be a closed irreducible algebraic subvariety of a mixed Shimura variety $\cS$. Then $\HL(X, \iota)_\atyp$ is not Zariski-dense in $X$.
\end{conjecture}

For a general period map $\varphi: X^\an \lo \Gamma \backslash \cD$ with a non-algebraic target, the picture above remains partially true. According to \cite{bakker-brunebarbe-tsimerman-pure-image}, \cite{bakker-brunebarbe-tsimerman-mixed-image}, the image of the period map $\varphi$ factorizes (uniquely up to unique isomorphism) as: 
\[ \varphi : X^\an \lo W^\an \hookrightarrow \Gamma \backslash D~,\] where $W$ (the {\em image} of the period map) is a (usually singular) irreducible quasi-projective variety, the algebraic morphism $X\lo W$ is dominant and $W^\an \hookrightarrow \Gamma \backslash \cD$ is a closed analytic immersion. To take care of fields of definition it will be convenient for us to consider the normalization $Y$ of $W$ rather than $W$ itself.

\begin{definition} \label{def0}
We call 
\[ \varphi : X^\an \stackrel{f^\an}{\lo} Y^\an \stackrel{\iota}{\lo} \Gamma \backslash D\]
the factorization of the period map $\varphi$, and $Y$ the {\em normalized image} of $\varphi$.
\end{definition}

Again, $\HL(X, \varphi)_\atyp = f^{-1}(\HL(Y, \iota)_\atyp)$, and the study of \Cref{ZP} for $X$ and $\varphi$ is reduced to the study of \Cref{ZP} for $Y$ and $\iota$.
\subsection{Results}

\cite{baldi-klingler-ullmo} in the pure case, generalized by \cite{BU} in the mixed case, proved the {\em geometric} part of \Cref{ZP}, see \Cref{reminder}. In this paper we use this geometric result and spreading arguments towards the full \Cref{ZP}. 

Our first result deals with the mixed Shimura case. Any mixed Shimura variety $\cS$ admits a canonical model $\cS_{\oQ}$ over $\oQ$. Any morphism of mixed Shimura variety coming from a morphism of mixed Shimura data is defined over $\oQ$. In particular, any quotient $\cS \twoheadrightarrow \cS'$ coming from a quotient mixed Shimura datum $(\bfG, \cD) \twoheadrightarrow (\bfG', \cD')$ is defined over $\oQ$; and any special subvariety of $\cS$ is defined over $\oQ$.

\begin{theorem}\label{main-shimura}
Let $X \stackrel{\iota}{\hookrightarrow} \cS$ be a Hodge-generic closed irreducible algebraic subvariety of a mixed Shimura variety $\cS$. If no non-trivial image $X'$ of $X$ in a quotient Shimura variety $\cS \twoheadrightarrow \cS'$ is defined over $\oQ$ (as a subvariety of $\cS'$), then \Cref{ZP-shimura} holds true for~$X$.
\end{theorem}

Let us illustrate this result in the pure case. Let $\cS$ be a pure Shimura variety, with Shimura datum $(\bfG, \cD)$. The decomposition of the adjoint group $\bfG^\ad$ into a product of simple factors $\bfG_1 \times \ldots \times \bfG_k$ induces a decomposition of (a finite quotient of) $\cS$ into $\cS_1 \times \ldots \times \cS_k$.

\begin{corollary} \label{corol0}
Let $\cS$ be a pure Shimura variety.  
Let $X \stackrel{\varphi}{\hookrightarrow} \cS$ be a Hodge generic closed irreducible complex subvariety. If its non-constant projections $X_i \subset \cS_{i}$, $1 \leq i \leq k$, are not defined over~$\oQ$ (as subvarieties of the $\cS_i$), then \Cref{ZP-shimura} holds true for $X$. In particular if $\cS$ is indecomposable (i.e. $\bfG^\ad$ is simple) and $X$ is not defined over $\oQ$, then \Cref{ZP} holds true for $X$.
\end{corollary}

\begin{remarque}
    The special case of \Cref{corol0} where $\cS= \cA_{g}$ the moduli space of principally polarized abelian varieties of dimension~$g$ was recently also proven by Barroero-Dill \cite[Theorem 1.1]{barroero-dill}.
\end{remarque}
 
Let us now turn to $\Z$VMHSs not necessarily of mixed Shimura type.
We will need the following notation.

\begin{definition}  \label{def1}
Let $X$ be an irreducible smooth quasi-projective algebraic variety over $\C$. 
\begin{enumerate}
\item A $\Z$VMHS $\bV$ on $X$ is said {\em geometric} if, possibly after replacing $X$ by a Zariski-dense open subset, there exists a flat surjective equisingular algebraic family $g:Z \lo X$ such that $\bV$ is a subquotient of the $\Z$VMHS $R^i g_* \Z$. A period map $\varphi: X^\an \rightarrow \Gamma \backslash D$ is said {\em geometric} if it is the period map of some geometric $\Z$VMHS on $X$. If $g: Z \lo X$ can be defined over a subfield $K \subset \C$, we say that $\bV$, resp. $\varphi$, can be {\em geometrically defined} over~$K$.
\item We say that a $\Z$VMHS $\bV$ on $X$ (or a period map $\varphi: X^\an \lo \Gamma \backslash D$) factorizes over a field $K \subset \C$ if the normalized image $Y$ of $\varphi$ defined in \Cref{def0} admits a model over $K$.
\item A {\em quotient} of a period map $\varphi: X^\an \lo \Gamma \backslash D$ with generic Hodge datum $(\bfG, D)$ is a period map $$\varphi': X^\an \stackrel{\varphi}{\lo} \Gamma \backslash \cD \twoheadrightarrow \Gamma' \backslash \cD'$$ for some  quotient Hodge datum $(\bfG, \cD) \twoheadrightarrow (\bfG', \cD')$. It is said to be {\em non-trivial} if its normalized image $Y'$ is not a point.
\item A geometric $\Z$VMHS $\bV$ on $X$ (resp. a period map $\varphi: X^\an \lo \Gamma \backslash D$) is said {\em absolute} if, for any algebraically closed field of $K \subset \C$ of $\bV$ over which $\bV$ (resp. $\varphi)$ can be geometrically defined, the special subvarieties of $X$ for $\bV$ (resp. $\varphi$) are defined over~$K$.
\end{enumerate}
\end{definition}

\begin{remarque}
If $\cS$ is a mixed Shimura variety and $Y \stackrel{\iota}{\hookrightarrow} \cS$ is a closed irreducible algebraic subvariety of $\cS$, then $\iota$ is absolute. It is expected that any $\bV$ of geometric origin is absolute, see \cite{voisin-absolute}, \cite{saito-schnell}, \cite{klingler-otwinowska-urbanik} for results in that direction.
\end{remarque}

Our second result in this note is the following.
\begin{theorem}\label{main-absolute}
Suppose that all $\Z$VMHSs of geometric origin are absolute.
Let $X$ be an irreducible smooth quasi-projective algebraic variety over $\C$ and let $\bV$ be a geometric $\Z$VMHS on $X$, with period map $\varphi$. If no non-trivial quotient of $\varphi$ factorizes over $\oQ$, then the Zilber-Pink \Cref{ZP} holds true for $X$ and $\bV$. 
\end{theorem}

\begin{theorem} \label{main-geometric}
    \Cref{ZP} holds true for $\Z$VMHS of geometric origin if and only if it holds true for $\Z$VMHS of geometric origin defined over $\overline{\Q}$.
\end{theorem}

\section{Geometric finiteness} \label{reminder}
In this section, we recall the finiteness theorem for atypical Hodge loci {\em of positive period dimension} recently established in the pure case \cite{baldi-klingler-ullmo}
and in the mixed case \cite{BU}.
\medskip

Let $X$, $\bV$, $(\bfG,\cD)$ be as in \Cref{introduction}, with period map
$\varphi:X^\an \rightarrow \Gamma\backslash \cD$.
The notion of special subvariety can be generalized, using the algebraic monodromy group instead of the generic Mumford-Tate group. Let $(\bfH,\cD_{\bfH})$ be the monodromy datum of the variation $\bV$. Thus $\bfH$ is the algebraic monodromy group of the local system underlying $\bV$, and $\cD_{\bfH} \subset \cD$ is the corresponding monodromy domain. According to a well-known result of Deligne and Andr\'e \cite[Theorem 1]{andre-mixed}, the semi-simple group $\bfH$ is normal in the derived group $\bfG^{\rm{der}}$ of $\bfG$. Up to a finite cover of $X$, the period map $\varphi$ factors through the monodromy special subvariety $\cGamma_{\bfH} \backslash \cD_{\bfH}$ of $\cGamma\backslash \cD_{\bfG}$: 
\[\varphi:X^\an\rightarrow \cGamma_{\bfH} \backslash \cD_{\bfH}\subset \cGamma\backslash \cD_{\bfG} .\] We refer to \cite{klingler}, \cite{klingler-otwinowska}, \cite{BBKT}\cite{gao-klingler} for more details.

The following result is proven in \cite{klingler-otwinowska} in the pure case, the proof generalizes easily to the mixed case:
\begin{proposition}
Let $Z\subset X$ be an irreducible algebraic subvariety. Let $(\bfH_Z,\cD_{\bfH_Z})$ be the monodromy datum of $\bV_{|Z}$. The following conditions are equivalent.
    \begin{enumerate}
        \item $Z$ is maximal among the irreducible subvarieties of $X$ with monodromy datum $(\bfH_Z,\cD_{\bfH_Z})$. 
        \item $Z^\an$ is an irreducible component of the preimage under $\varphi$ of a monodromy special variety $\cGamma_{\bfH_Z}\backslash \cD_{\bfH_Z}$. 
    \end{enumerate}

If $Z$ satisfies one of the conditions above, then we say that $Z$ is weakly (or monodromy) special.
\end{proposition}

One easily checks that any special subvariety of $X$ is monodromy special, see \cite[7.1]{klingler}.

\begin{definition} \label{mono-atyp}
Let $Z\subset X$ be a monodromy special subvariety for $\bV$, with monodromy data $(\bfH_Z,\cD_{\bfH_Z})$. 
We say that $Z$ is monodromy atypical if:
\[\codim_{\varphi(X^\an)}(\varphi(Z^\an))<\codim_{\cD_{\bfH}}(\cD_{\bfH_{Z}})~.\]
\end{definition}

\begin{lemma} \label{atyp implies monod atyp}
    Let $Z \subset X$ be an atypical (resp. a maximal atypical) special subvariety for $\bV$. Then $Z$ is a monodromy atypical (resp. a maximal monodromy atypical) subvariety of $S$ for $\bV$. 
\end{lemma}

\begin{proof}
    When $\bV$ is pure this is \cite[Lemma 5.6]{baldi-klingler-ullmo}, and \cite[Rem. 4.1]{baldi-klingler-ullmo} for the maximality. The proof there extends verbatim to the mixed case.
\end{proof}

The following theorem solves the geometric part of Zilber--Pink conjecture. It was established in the pure case by Baldi-Klingler-Ullmo \cite[Theorem 6.1]{baldi-klingler-ullmo} and in the mixed case by Baldi-Urbanik \cite[Theorem 7.1]{BU}. Recall that by a {\em family of subvarieties of $X$} we mean a pair $(f: Z \lo Y, \pi:Z \lo X)$ where $f$ is a flat surjective morphism of irreducible varieties, and $\pi:Z \rightarrow X$ is an algebraic morphism whose restriction to each fiber of $f$ is a closed embedding. 

\begin{theorem}\label{finiteness}
 There are only finitely many families $(f_i: Z_i \lo Y_i, \, \pi_i: Z_i \lo X)$, $i \in \Sigma$, of maximal monodromy atypical subvarieties of $X$ for $\bV$. 

More precisely: there exists a finite set $\Sigma$ of triples $(\bfG_i,\cD_{i},\bfN_i)_{i\in \Sigma}$, where $(\bfG_i,\cD_{i})$ is a sub-Hodge datum of the generic Hodge datum $(\bfG,\cD_{\bfG})$, $\bfN_i$ is a normal subgroup of $\bfG_i$ whose reductive part is semisimple, and the following property holds: for each maximal monodromy atypical subvariety $Y\subset X$, there exist $i\in \Sigma$ and $y\in \cD_i$ for which $\cD_{\bfH_{Y}}=\bfN_i(\R)^+\bfN_i(\C)^u\cdot y$ \textup{(}up to the action of $\cGamma$\textup{)} and $Y$ is isomorphic to a fiber of $f_i$.
\end{theorem}

\section{Reduction to the atypical Hodge locus of zero period dimension}
Let $X$ be a smooth quasi-projective algebraic variety defined over $\C$ and $\bV$ a $\Z$VMHS on $X$. Let $(\bfG, \cD)$ be the generic Hodge datum of $\bV$ and $\varphi: X^\an \rightarrow \Gamma \backslash \cD$ the associated period map. 

\begin{definition} (see \cite{klingler-otwinowska})
    Let $Z$ be a special subvariety of $X$ for $\bV$. We say that $Z$ has zero period dimension if $\varphi(Z^\an)$ is a point; otherwise we say that $Z$ has positive period dimension.
\end{definition}

We denote by $\HL(X, \varphi)_{\atyp,\pos}$ the countable union of atypical special subvarieties of positive period dimension, and by $\HL(X, \varphi)_{\atyp,\, \zero}$ the union of atypical special subvarieties of zero period dimension, thus obtaining a decomposition: 
\[\HL(X, \varphi)_{\atyp}=\HL(X, \varphi)_{\atyp,\pos}\cup \HL(X, \varphi)_{\atyp,\, \zero}~.\]

\begin{definition}
    We define $$\HL(X, \varphi)_{\maxatyp,\, \zero}:= \HL(X, \varphi)_{\atyp,\, \zero} \setminus\HL(X, \varphi)_{\atyp,\pos}$$ the union of maximal atypical special subvarieties of zero period dimension (i.e. those not contained in an atypical special subvariety of positive period dimension).
\end{definition}

\begin{proposition} \label{reduction}
    \Cref{ZP} holds true for $X$ and $\bV$ if and only if for any quotient Hodge datum $(\bfG, \cD)\twoheadrightarrow (\bfG', \cD')$, $\HL(X, \varphi')_{\maxatyp, \, \zero}$ is not Zariski-dense in $X$.
\end{proposition}

\begin{proof}

    As $\HL(X, \varphi')_{\maxatyp, \, \zero} \subset \HL(X, \varphi)_{\atyp}$, \Cref{ZP} certainly implies the right hand side condition. Conversely, let us suppose that for any quotient Hodge datum $(\bfG, \cD)\twoheadrightarrow (\bfG', \cD')$, $\HL(X, \varphi')_{\maxatyp, \, \zero}$ is not Zariski-dense in $X$.

\medskip
   Let $Z\subset X$ be a maximal atypical special subvariety of $X$ for $\bV$. By \Cref{atyp implies monod atyp}, $Z\subset X$ is a maximal monodromy atypical weakly special subvariety of $X$ for $\bV$. Thus $Z$ is a fiber of one of the finitely many families $f_i:Z_i \rightarrow Y_i$, $i \in \Sigma$, appearing in \Cref{finiteness}. For simplicity we will call such a fiber an atypical special fiber of $f_i$. We are reduced to showing that the union of the atypical special fibers of $f_i$, $i \in \Sigma$, is not Zariski-dense in $X$.

\medskip
Let us say that a family $(f, \pi)$ of subvarieties of $X$ is dominant if $\pi$ is. 
   If $f_i:Z_i \rightarrow Y_i$ is not dominant, then a fortiori the union of all atypical special fibers of $f_i$, which is contained in the image of $\pi_i: Z_i \rightarrow X$, is not Zariski-dense in $X$. Let $\Sigma^\dom\subset \Sigma$ denote the finite set of dominant families. We are thus reduced to showing that the union of the atypical special fibers of $f_i:Z_i \rightarrow Y_i$, $i \in \Sigma^\dom$, is not Zariski-dense in $X$. 
   
\medskip

   Let $i \in \Sigma^\dom$. Then necessarily $(\bfG_i, \cD_i) = (\bfG, \cD)$ (otherwise the image of $\pi_i$ would be contained in the Hodge locus of $X$ for $\bV$ with Hodge subdatum $(\bfG_i,\cD_i)$, a contradiction to the fact that $\pi_i$ is dominant). Let $(\bfG'_i, \cD'_i)$ be the quotient Hodge datum of $(\bfG, \cD)$ by $\bfN_i$ and let $$\varphi'_i: X^\an \stackrel{\varphi}{\rightarrow} \Gamma \backslash \cD \twoheadrightarrow {\Gamma'_i} \backslash \cD'_i$$ be the associated quotient period map. As $\bfN_i$ is the algebraic monodromy group of the fibers of $f_i$, the atypical special fibers of $f_i$ are in bijection with the maximal atypical special subvarieties of zero period dimension of $X$ for $\varphi'_i$. Hence the union of all atypical special fibers of the morphism $f_i$ coincides with $\HL(X, \varphi'_i)_{\maxatyp, \, \zero}$. It follows from our hypothesis that this is not Zariski-dense in $X$. Thus \Cref{ZP} holds true. 
   
\end{proof}

\section{Proof of \Cref{main-shimura}}

In view of \Cref{reduction}, \Cref{main-shimura} follows from:

\begin{theorem} \label{shimura-zero}
    Let $X\stackrel{\iota}{\hookrightarrow} \cS$ be a Hodge generic irreducible smooth quasi-projective algebraic subvariety of a mixed Shimura variety $\cS$. If no non-trivial image $X'$ of $X$ in a quotient Shimura variety $\cS \twoheadrightarrow \cS'$ is defined over $\oQ$, then $\HL(X, \iota)_{\maxatyp, \, \zero}$ is not Zariski-dense in $X$.
\end{theorem}

\begin{proof}[Proof of \Cref{shimura-zero}]

The proof is by induction on the minimal transcendence degree~$n$ over $\oQ$ of an algebraically closed field~$K$ of definition of $X$. By hypothesis $n\geq 1$. Let $\Lambda:= \HL(X, \varphi)_{\maxatyp, \, \zero}$ be the set of maximal atypical special points of $X$. Suppose by contradiction that $\Lambda$ is Zariski-dense in~$X$.

\medskip
Choose $L$ an algebraically closed field, $\oQ \subset L \subset K$, such that $K$ has transcendence degree~$1$ over $L$. This is possible as $n \geq 1$. 
Let $\Phi: \cX \lo \cS$ be a spread-out over $L$ of the morphism $\iota:X \hookrightarrow \cS$. As $\cS$ is defined over $\oQ$, without loss of generality we can assume that: 
\begin{enumerate}
    \item there exists an equisingular morphism of complex algebraic varieties $\pi:\cX \rightarrow \cB$, defined over~$L$; 
    \item the base $\cB$ is irreducible, smooth over $L$, of dimension~$1$;
    \item the variety $X$ is the fiber of $\pi$ at a $K$-point of $B$, and is $L$-Zariski-dense in $\cX$; and $\Phi$ restricted to $X$ coincides with $\iota$;
    \item the restriction of $\Phi: \cX \lo \cS$ to each fiber of $\pi$ is a closed immersion.
\end{enumerate}

\medskip
In particular the normalized image~$Y$ of the period map $\Phi: \cX \lo \cS$ has complex dimension $\dim_\C X$ or $\dim_\C X +1$. In the first case, $Y$ is irreducible, contains $X$ and has the same dimension as $X$, hence $Y=X$. As $\Phi$ and $Y$ are defined over $L$, it follows that $X$ is defined over $L$, a contradiction to the minimality of $n$. Thus $\dim_\C Y = \dim_\C X +1$, i.e., $\Phi$ is generically finite.

\medskip
For each $x \in \Lambda$, let $S_{x} \subset \cS$ be the unique special subvariety whose generic Mumford-Tate group is the one of $x$ and such that $x$ is an irreducible component of $X \cap S_{x}$ (notice that although $x$ is a point, the special subvariety $S_x$ can be positive dimensional). We denote by $Z_x \subset \cX$ one irreducible component of $\Phi^{-1}(S_x)$ containing~$x$. Thus $Z_{x} \subset \cX$ is a special subvariety of $\cX$ for $\Phi$, and $Z_{x}$ is defined over~$L$. As $\dim_\C \cX = \dim_\C X +1$, the special subvariety $Z_{x}$ of $\cX$ is atypical as soon as $\Phi(Z_{x})$ is not a point (namely~$x$).

\medskip
Let $\Lambda' \subset \Lambda$ be the subset of points $x$ such that $\Phi (Z_{x})=x$. If $x \in \Lambda'$, it follows that $x = \Phi(Z_{x})$ is a point of $\cS$ defined over $L$, as both $Z_x$ and $\Phi$ are defined over $L$. Thus the Zariski-closure of $\Lambda'$ in $X$ is defined over~$L$. It follows that $\Lambda'$ is not Zariski-dense in $X$, as $X$ cannot be defined over~$L$.

\medskip
Replacing $\Lambda$ by $\Lambda- \Lambda'$ if necessary, we can thus without loss of generality assume that each $Z_{x}$ is an atypical subvariety of $\cX$ for $\Phi$, of positive period dimension. According to \Cref{finiteness}, any such $Z_{x}$ is contained in a fiber of one of the finitely many families $(f_i: Y_i \lo B_i, \pi_i: Y_i \lo \cX)$, $i \in \Sigma$, of maximal monodromy atypical subvarieties of $\cX$ for $\Phi$. Let $\Lambda_i \subset \Lambda$, $i \in \Sigma$, be the subset of $x \in \Lambda$ such that $Z_{x}$ is contained in a fiber of $f_i$.

\medskip
Consider the Zariski-closure $T \subset \cX$ of 
$$\bigcup_{i \in (\Sigma - \Sigma^\dom)} \bigcup_{x \in \Lambda_i} Z_{x},$$
where $\Sigma^\dom$ denotes, as in the proof of \Cref{reduction}, the collection of $i \in \Sigma$ such that $\pi_i$ is dominant.
The variety $T$ is, by definition of $\Sigma^\dom$, a strict closed algebraic subvariety of $\cX$. It is also defined over $L$ as each $Z_{x}$ is. As $X \subset \cX$ is $L$-Zariski-dense in $\cX$, it follows that $T \cap X$, hence also its subset $\bigcup_{i \in(\Sigma - \Sigma^\dom)} \bigcup_{x \in \Lambda_i} \{x\}$, is not Zariski-dense in~$X$. 

\medskip Therefore, there exists $i \in \Sigma^\dom$ such that $\Lambda_i$ is Zariski-dense in $X$. Without loss of generality, we can replace $\Lambda$ by such a $\Lambda_i$. Let us write $(f', \pi', \bfN):= (f_i, \pi_i, \bfN_i)$. As in the proof of \Cref{reduction}, $(\bfG_i, \cD_i) = (\bfG, \cD)$. Let $(\bfG', \cD')$ be the quotient Shimura datum of $(\bfG, \cD)$ by $\bfN$ and let \[\Phi': \cX \stackrel{\Phi}{\rightarrow} \cS \stackrel{p}{\twoheadrightarrow} \cS'\] be the associated quotient period map. Let $X':= p(X) \subset \cX':= \Phi'(\cX) \subset \cS'$. As the group $\bfN$ is the algebraic monodromy group of the fibers of $f'$, $\Phi'(Z_{x})=p(x)$ for each $x \in \Lambda$. Thus $\Phi'(Z_{x})=p(x)$ is a maximal atypical special point $x'$ of $X'\hookrightarrow \cS'$. It is defined over~$L$ as $Z_x$ and $\Phi'$ are. Let $\Lambda' \subset \cS'$ be the set of such points. The image $\cX'= \Phi'(\cX)$ is also defined over $L$, as $\cX$ and $\Phi'$ are. As $\Lambda$ is Zariski-dense in $X$ and $X$ is $L$-Zariski-dense in $\cX$, it follows that $\Lambda'$ is Zariski-dense both in $\cX'$ and $X'= \Phi'(X)= p(X)$. Thus the quotient $X' = \cX'$ of $X$ is a subvariety of $\cS'$ defined over~$L$. It is non-trivial: indeed $X$ contains a Hodge-generic point of $\cS$, hence $X'$ contains a Hodge-generic point of $\cS'$ and it also contains a dense collection of special points. 

\medskip
If $n=1$, this means that $X'$ is defined over $L= \oQ$, a contradiction to our assumption that $X$ has no non-trivial quotient defined over $\oQ$. 

\medskip
If $n>1$: as any quotient of $X'$ is a quotient of $X$, and $X$ has no non-trivial quotient defined over~$\oQ$, $X'$ has also no non-trivial quotient defined over~$\oQ$. As $X'$ is defined over a field $L$ of transcendence degree $n-1$ over $\oQ$, our induction hypothesis applied to $X'$ implies that $\Lambda'$ is finite; hence its Zariski-closure $X'$ has dimension zero: a contradiction to the non-triviality of $X'$.

\end{proof}


\section{Proof of \Cref{main-absolute}}

\subsection{Preliminaries}

As for the proof of \Cref{main-shimura}, spreading is the main tool in the proof of \Cref{main-absolute}. This requires the $\Z$VMHSs we consider to be geometric. In this section, we discuss some of the consequences of \Cref{def1} that we will need. 
First, we need to extend a bit \Cref{def1}.

\begin{definition} \label{def2}
   Given a non-necessarily geometric $\Z$VMHS $\bV$ on $X$, we say that it is defined over a field $K \subset \C$ if (possibly after replacing $X$ by a Zariski-dense open subset) the variety $X$, the filtered module with integrable connection $(\cV, F^\bullet, \nabla)$ on $X$ corresponding to $\bV$ under the Deligne-Riemann-Hilbert correspondence, its weight filtration $W_\bullet$ and the graded-polarisations $Q_i : \Gr^W_i \times \Gr^W_i \lo \cO_X$ are all defined over~$K$. 
   
   We say that a period map $\varphi: X^\an \lo \Gamma \backslash \cD$ is defined over a field $K \subset \C$ if it is the period map of a $\Z$VMHS defined over $K$.
\end{definition}

\begin{lemma} \label{subquotient}
Let $X$ be a smooth irreducible quasi-projective algebraic variety over $\C$, and $\bV$ a $\Z$VMHS on $X$ defined over~$K$. Then the period map of any subquotient of $\bV$ is defined over~$K$.
\end{lemma}

\begin{proof}
Let $\varphi$ be the period map of a sub-$\Z$VMHS $\bW \subset \bV$. In general $\bW$ is defined only over an extension $M$ of~$K$. Notice however that the isotypic component 
\[\sum_{\tau\in \Hom(\bW, \bV)} \tau(\bW) \subset \bV\] 
of $\bW$ in $\bV$ is defined over~$K$, as the corresponding filtered module with integrable connection, weight filtration and polarization are stable under $\Aut(M/K)$. Moreover, this isotypic component has the same period map $\varphi$ as $\bV$. It follows that $\varphi$ is defined over~$K$.

\medskip

Arguing with duals, we deduce in the same way that the period map of a quotient of a $\Z$VMHS defined over~$K$ is also defined over~$K$. The result follows.
\end{proof}

\begin{corollary} \label{geometrically defined= defined}
Let $X$ be a smooth irreducible quasi-projective algebraic variety over $\C$, and $\varphi:X^\an \lo \Gamma \backslash \cD$ a geometric period map geometrically defined over~$K$ in the sense of \Cref{def1}(1). Then $\varphi$ is defined over~$K$ in the sense of \Cref{def2}.
\end{corollary}
\begin{proof}
    By definition $\varphi$ is the period map of a subquotient $\bV$ of some $\Z$VMHS $R^i g_*\Z$, for $g: Z \lo X$ a flat surjective equisingular morphism defined over~$K$. 
    
    It is classical that $R^i g_*\Z$ is defined over~$K$ in the sense of \Cref{def2}. See \cite[Section 2]{urbanik-absolute} for a detailed proof in the pure case (in the more general setting of what Urbanik calls motivic $\Z$VHS); the proof in the mixed case is similar. 

    The fact that $\varphi$ is defined over~$K$ thus follows immediately from \Cref{subquotient}.
    \end{proof}

\begin{proposition} \label{fields of definition}
    Let $X$ be a smooth irreducible quasi-projective algebraic variety over $\C$, and \[\varphi: X^\an\stackrel{f^\an}{\lo} Y^\an \stackrel{\iota}{\lo} \Gamma \backslash \cD\] a factorized period map as in \Cref{def0}.
    \begin{enumerate}
     \item If $\varphi$ is geometric in the sense of \Cref{def0}, then $\iota: Y^\an  \lo \Gamma \backslash \cD$ is geometric. 
     
     \item If $\varphi$ is defined over~$K \subset \C$, then $Y$, $f: X\lo Y$, and $\iota$ are also defined over $K$.   

\item If $\varphi$ is geometrically defined over~$K$ then $\iota$ is geometrically defined over~$K$.
\end{enumerate}
\end{proposition}

\begin{proof} 
Suppose that $\varphi$ is geometric. Replacing $X$ by a Zariski-dense open subset if necessary, there exists $g: Z \lo X$ a flat surjective equisingular morphism such that $\varphi$ is the period map of a subquotient $\bV$ of $R^ig_* \Z$. Replacing $X$ and $Y$ by Zariski-dense open subsets if necessary, we can assume without loss of generality that $f$ is smooth, hence $f \circ g: Z \lo Y$ is flat surjective equisingular. The map $\iota: Y \lo \Gamma \backslash \cD$ is the period map of the $\Z$VMHS $f_* (\bV)$ on $Y$. But $f_* (\bV)$ is a subquotient of $R^i(g \circ f)_*\Z$, as the Leray spectral sequence for $f \circ g$ degenerates at the $E_2$ page. This shows that $\iota$ is geometric.

Suppose that $\varphi$ is defined over an algebraically closed field $K \subset \C$. When $K = \oQ$ and $\varphi$ is pure, Urbanik \cite[Theor. 1.6]{urbanik-image-model} shows that both $Y$ and $f: X \lo Y$ are defined over $\oQ$ (this is where we need to work with the normalized image of $\varphi$ rather than its image). His proof works verbatim for any algebraically closed field $K \subset \C$, and can be easily adapted to the mixed case. Thus $Y$ and $f:X \lo Y$ are defined over~$K$ as soon as $\varphi$ is. Again, $\iota: Y \lo \Gamma \backslash \cD$ is the period map of the $\Z$VMHS $f_* (\bV)$ on $Y$, which is then defined over~$K$ as both $f$ and $\bV$ are. Thus $\iota$ is defined over~$K$.

Suppose that $\varphi$ is geometrically defined over~$K$. In particular, by \Cref{geometrically defined= defined}, $\varphi$ is defined over~$K$. Hence both $Y$ and $f: X \lo Y$ are defined over $K$. The map $\iota: Y \lo \Gamma \backslash \cD$ is the period map of the $\Z$VMHS $f_* (\bV)$ on $Y$, which is a subquotient of the $\Z$VMH $R^i(g \circ f)_*\Z$ defined over~$K$. Hence $\iota$ is geometrically defined over~$K$.
\end{proof}

\begin{lemma} \label{geometry3}
Let $X$ be a smooth irreducible quasi-projective variety over $\C$, $\varphi: X^\an \lo \Gamma \backslash \cD$ a period map on $X$, and 
\[\varphi' : X^\an \stackrel{\varphi}{\lo} \Gamma \backslash \cD \twoheadrightarrow \Gamma' \backslash \cD'\] a quotient of $\varphi$.
If $\varphi$ is geometric, then $\varphi'$ is also geometric. If moreover $\varphi$ is defined over an algebraically closed field $K \subset \C$, then $\varphi'$ too.
\end{lemma}

\begin{proof}

Suppose that $\varphi$ is geometric. After replacing $X$ by a Zariski-dense open subset if necessary, there exists $g: Z \lo X$ a flat surjective equisingular map such that $\varphi$ is the period map of a subquotient $\bV$ of $R^ig_*\Z $. Thus $\textnormal{End}\, \bV$ is also geometric, as a subquotient of $R^{2i}(g \times g)_*\Z$. 

Let $\Ad: \bfG \lo \GL(\mathfrak{g})$ be the adjoint representation of the generic Mumford-Tate group of $\bV$, and $\bV_{\mathfrak{g}} \subset \textnormal{End}\, \bV$ the associated $\Z$VMHS. Thus $\bV_{\mathfrak g}$ is geometric.

Let $\bfN \subset \bfG$ be the kernel of the projection $\bfG \twoheadrightarrow \bfG'$, and $\mathfrak{n}$ its Lie algebra. As $\mathfrak{n}$ is an ideal in $\mathfrak{g}$, it is stable under the representation $\textnormal{Ad}: \bfG \lo \GL(\mathfrak{g})$, hence defines a sub-$\Z$VMHS $\bV_{\mathfrak{n}} \subset  \bV_{\mathfrak{g}}$. The quotient $\bV_{\mathfrak{g}'}= \bV_{\mathfrak{g}}/\bV_{\mathfrak{n}}$ is thus geometric. But its period map is $\varphi'$, hence $\varphi'$ is geometric.

\medskip
If $\varphi$ is defined over an algebraically closed subfield $K \subset \C$, it follows from the previous construction and \Cref{fields of definition} that $\varphi'$ too.

\end{proof}

\subsection{The proof}
\begin{proof}[\unskip\nopunct]
We assume that all geometric $\Z$VMHS are absolute.
Consider the collection $\cC$ of factorized geometric period maps 
\[\varphi: X^\an\stackrel{f^\an}{\lo} Y^\an \stackrel{\iota}{\lo} \Gamma \backslash \cD,\] as defined in \Cref{def0}, such that no non-trivial quotient of $\varphi$ factorizes over $\oQ$. 

\medskip
To any $\varphi \in \cC$ we can associate its transcendence degree $\td(\varphi)$: the minimal transcendence degree over $\oQ$ of an algebraically closed field $K \subset \C$ over which $\varphi$ is geometrically defined. Hence the assumption that $\varphi$ does not factorize over~$\oQ$ implies that the transcendence degree of any $\varphi \in \cC$ is at least~$1$.

\medskip
We want to show that for any $\varphi \in \cC$, the set 
$\HL(Y, \iota)_{\maxatyp, \, \zero}$ of maximal atypical special points of $Y$ is not Zariski-dense in $Y$. Suppose by contradiction that this is not the case, and let $\varphi\in \cC$ be a counterexample of minimal transcendence degree $n>1$. Let $K \subset \C$ be an algebraically closed field of definition of some $g:Z \lo X$, of transcendence degree~$n$ over $\oQ$, such that $\varphi$ is the period map of a subquotient $\bV$ of $R^ig_* \Z$. In particular, it follows from \Cref{fields of definition} that $f: X \lo Y$ is defined over~$K$.

\medskip
Let $L\subset K$ be an algebraically closed subfield such that $K$ has transcendence degree $1$ over $L$. Let $$\cZ\stackrel{G}{\rightarrow} \cX \stackrel{\pi}{\rightarrow} \cB$$ be a spread-out of $g: Z \lo X$, where $\cB$ is a smooth irreducible curve defined over $L$. Without loss of generality, we can assume 
\begin{enumerate}
    \item The morphisms $\pi:\cX\rightarrow \cB$ and $G: \cZ\rightarrow \cX$ are surjective flat equisingular, defined over $L$.
    \item The variety $X$ identifies with (the base change to $\C$ of) a fiber $\pi^{-1}(\{b\})$ of $\pi$ over a $K$-point $b$ which is $L$-Zariski dense in $\cX$. In particular, $X$ is $L$-Zariski dense in $\cX$.
\end{enumerate}

Without loss of generality we can assume that the $\Z$VMHS $\bV$ subquotient of $R^ig_* \Z$ extends to a subquotient $\tilde{\bV}$ of $R^i G_* \Z$. Otherwise the exceptional subquotient $\bV$ which exists on the hypersurface $X$ of $\cX$ but not on a very general point of $\cX$ makes $X$ a special subvariety of $\cX$ for $R^i G_* \Z$. Under our assumption that all geometric $\Z$VMHS are absolute, it follows that $X\subset \cX$ is a closed subvariety defined over~$L$. This is a contradiction to~$(2)$.

\medskip
Let $(\cG_{\cX},\cD_{\cX})$ be the generic Hodge datum of the $\Z$VMHS $\tilde{\bV}$ on $\cX$, and \[\Phi: \cX^\an \stackrel{F^\an}{\lo} \cY^\an \stackrel{\tilde{\iota}}{\lo} \Gamma_\cX \backslash \cD_{\cX}\] its factorized period map. Thus, according to \Cref{fields of definition}, $F: \cX \lo \cY$ is defined over $L$. Moreover, by construction, $\Phi_{|X^\an} = \varphi$, $Y \subset \cY$, and $(\cG,\cD)\subset (\cG_{\cX},\cD_{\cX})$ is a sub-Hodge datum. 
Suppose that $(\cG_{\cX},\cD_{\cX}) \neq (\cG,\cD)$. Thus $X \subset \HL(\cX, \Phi)$ is a special subvariety of $\cX$ for $\Phi$. Once again, our assumption that all geometric $\Z$VMHS are absolute implies that $X\subset \cX$ is defined over~$L$, a contradiction to~$(2)$. Thus $(\cG_{\cX},\cD_{\cX})= (\cG,\cD)$. 

\medskip
There are two cases:

\medskip
{\bf Case 1:} $\cY=Y$. Thus $\tilde{\iota}= \iota$. But then $\Phi \in \cC$ has transcendence degree at most $n-1$ and $\HL(\cY, \tilde{\iota})_{\maxatyp, \, \zero}= \HL(Y, \iota)_{\maxatyp, \, \zero}$ is Zariski-dense in $\cY$. This is a contradiction to the minimality of~$n$ as the transcendence degree of a counterexample.

\medskip
{\bf Case 2:} $\dim(\cY)=\dim(Y)+1$. 
We argue as in the Shimura case.
For $y\in\Lambda$, let $S_y\subset \Gamma\backslash \cD$ be the unique Hodge subvariety such that $\{y\}$ is a component of the preimage of $S_y$ by $\varphi$. Let $T_y=\tilde{\iota}^{-1}(S_y)^{0}$ be one of the corresponding special subvarieties of $\cY$ for $\tilde{\iota}$ passsing through $y\in Y^{\an}\subset \cY^{\an}$. 

Without loss of generality, we can again assume that $T_y$ is positive dimensional for all $y \in \Lambda$. Otherwise $T_y = y$ for $y \in \Lambda' \subset \Lambda$ still Zariski-dense in $Y$. As $T_y$ is defined over~$L$ under our assumption that all geometric $\Z$VMSHs are absolute, it would follow that the Zariski-closure $Y = \bigcup_{y \in \Lambda} \{y\}$ is a subvariety of $\cY$ defined over~$L$. But then $X$, which is an irreducible component of $F^{-1}(Y)$, is also defined over~$L$ as $F$ and $Y$ are. This is again a contradiction to~$(2)$. 

It follows that for all $y \in \Lambda$, $T_y$ is an atypical positive dimensional subvariety of $\cY$ for $\tilde{\iota}$. Moreover, the union of $T_y$ for $y\in\Lambda$ is Zariski dense in $\cY$. Again, according to \Cref{mono-atyp}, we can assume that there exists a family $(f':W'\rightarrow B',\pi':W'\rightarrow \cX, \bfN)$ of maximal monodromy atypical subvarieties of $\cX$ for $\Phi$ such that $T_y$ is contained in a fiber of $f_i$ for all $y \in \Lambda$. The group $\bfN$ is the algebraic monodromy group of the fibers of $f'$.

\medskip
Let $(\bfG', \cD')$ be the quotient Hodge datum of $(\bfG, \cD)$ by $\bfN$. We thus obtain the following diagram (the left hand side is in the algebraic category, while the right hand side is in the analytic one, for simplicity we did not write the symbol $\cdot^\an$ where needed).
\[
\xymatrix{
Z\ar[dr]^g \ar@{^(->}[r]& \cZ \ar[dr]^G &  &\\
&X \ar@{^(->}[r] \ar[d]_f &\cX \ar[d]_F \ar[rd]^{\Phi} & \\
& Y \ar[d] \ar@{^(->}[r]  &\cY \ar[r]^{\tilde{\iota}} \ar[d]_{q} & \Gamma \backslash D \ar@{->>}[d]^p \\
&Y'\ar@{^(->}[r] & \cY' \ar[r]_{\tilde{\iota '}} &\Gamma' \backslash D'.
}
\]
Here $\cY'$ is the normalized image of the period map $p \circ \Phi$. It follows from \Cref{geometry3} that $p \circ \Phi$ is geometric, defined over the same field $L$ as $\Phi$. In particular $q: \cY \lo \cY'$ is defined over~$L$ according to \Cref{fields of definition} 

\medskip Notice that for each $y \in \Lambda$, $q(T_{y})= q(y)$ is a maximal atypical special point $y'$ of $\cY'^\an \stackrel{\tilde{\iota'}}{\lo} \Gamma'\backslash \cD'$. It is defined over~$L$ as $T_y$ and $q$ are. Let $\Lambda' \subset \cY'$ be the set of such points. As $\Lambda$ is Zariski-dense in $Y$ and $Y$ is $L$-Zariski-dense in $\cY$, it follows that $\Lambda'$ is Zariski-dense both in $\cY'$ and $Y'= q(Y)$. Thus $Y' = \cY'$, $ \iota' =\tilde{\iota'}$, and the non-trivial quotient $\iota'$ of $\varphi$ is defined over~$L$. 

\medskip
If $n=1$, this is a contradiction to the hypothesis that $\varphi$ has no non-trivial quotient defined over $\oQ$. 

\medskip
If $n>1$: as any quotient of $\iota'$ is a quotient of $\varphi$, and $\varphi$ has no non-trivial quotient defined over~$\oQ$, $\iota'$ has also no non-trivial quotient defined over~$\oQ$. Thus $\iota' \in \cC$. But $\Lambda'$, in particular $\HL(Y', \iota')_{\maxatyp, \zero}$ is Zariski-dense in $Y'$. As $\iota'$ is defined over the field $L$ of transcendence degree $n-1$ over $\oQ$, this is a contradiction to the minimality of~$n$ as the transcendence degree of a counterexample.

\end{proof}

\section{Proof of \Cref{main-geometric}}

\begin{proof}[\unskip\nopunct]
This follows immediately from the proof of \Cref{main-absolute}.
\end{proof}

\bibliographystyle{alpha}
\bibliography{bibliographie}

\begin{thebibliography}{BBKT24}

\bibitem[And92]{andre-mixed}
Yves Andr\'e.
\newblock Mumford-{T}ate groups of mixed {H}odge structures and the theorem of
  the fixed part.
\newblock {\em Compositio Math.}, 82(1):1--24, 1992.

\bibitem[BBKT24]{BBKT}
Benjamin Bakker, Yohan Brunebarbe, Bruno Klingler, and Jacob Tsimerman.
\newblock Definability of mixed period maps.
\newblock {\em J. Eur. Math. Soc. (JEMS)}, 26(6):2191--2209, 2024.

\bibitem[BBT23a]{bakker-brunebarbe-tsimerman-pure-image}
Benjamin Bakker, Yohan Brunebarbe, and Jacob Tsimerman.
\newblock o-minimal {GAGA} and a conjecture of {Griffiths}.
\newblock {\em Invent. Math.}, 232(1):163--228, 2023.

\bibitem[BBT23b]{bakker-brunebarbe-tsimerman-mixed-image}
Benjamin Bakker, Yohan Brunebarbe, and Jacob Tsimerman.
\newblock Quasi-projectivity of images of mixed period maps.
\newblock {\em J. Reine Angew. Math.}, 804:197--219, 2023.

\bibitem[BD24]{barroero-dill}
Fabrizio {Barroero} and Gabriel~Andreas {Dill}.
\newblock {Hecke orbits and the Mordell-Lang conjecture in distinguished
  categories}, January 2024.

\bibitem[BKT20]{bakker-klingler-tsimerman}
Benjamin Bakker, Bruno Klingler, and Jacob Tsimerman.
\newblock Tame topology of arithmetic quotients and algebraicity of {H}odge
  loci.
\newblock {\em J. Amer. Math. Soc.}, 33(4):917--939, 2020.

\bibitem[BKU24]{baldi-klingler-ullmo}
Gregorio Baldi, Bruno Klingler, and Emmanuel Ullmo.
\newblock On the distribution of the {Hodge} locus.
\newblock {\em Invent. Math.}, 235(2):441--487, 2024.

\bibitem[BP09]{brosnan-pearlstein}
Patrick Brosnan and Gregory~J. Pearlstein.
\newblock The zero locus of an admissible normal function.
\newblock {\em Ann. of Math. (2)}, 170(2):883--897, 2009.

\bibitem[BPS10]{brosnan-pearlstein-schnell}
Patrick Brosnan, Gregory Pearlstein, and Christian Schnell.
\newblock The locus of {H}odge classes in an admissible variation of mixed
  {H}odge structure.
\newblock {\em C. R. Math. Acad. Sci. Paris}, 348(11-12):657--660, 2010.

\bibitem[BU24]{BU}
Gregorio {Baldi} and David {Urbanik}.
\newblock {Effective atypical intersections and applications to orbit
  closures}.
\newblock {\em arXiv e-prints}, page arXiv:2406.16628, June 2024.

\bibitem[CDK95]{cattani-deligne-kaplan}
Eduardo Cattani, Pierre Deligne, and Aroldo Kaplan.
\newblock On the locus of {H}odge classes.
\newblock {\em J. Amer. Math. Soc.}, 8(2):483--506, 1995.

\bibitem[GK24]{gao-klingler}
Ziyang Gao and Bruno Klingler.
\newblock The {A}x-{S}chanuel conjecture for variations of mixed {H}odge
  structures.
\newblock {\em Math. Ann.}, 388(4):3847--3895, 2024.

\bibitem[Kli17]{klingler}
Bruno Klingler.
\newblock Hodge loci and atypical intersections: conjectures.
\newblock Preprint, {arXiv}:1711.09387 [math.{AG}] (2017), 2017.

\bibitem[KO21]{klingler-otwinowska}
Bruno Klingler and Anna Otwinowska.
\newblock On the closure of the {Hodge} locus of positive period dimension.
\newblock {\em Invent. Math.}, 225(3):857--883, 2021.

\bibitem[KOU23]{klingler-otwinowska-urbanik}
Bruno Klingler, Anna Otwinowska, and David Urbanik.
\newblock On the fields of definition of {Hodge} loci.
\newblock {\em Ann. Sci. {\'E}c. Norm. Sup{\'e}r. (4)}, 56(4):1299--1312, 2023.

\bibitem[Pin05]{Pink}
Richard Pink.
\newblock A combination of the conjectures of {M}ordell-{L}ang and
  {A}ndr\'e-{O}ort.
\newblock In {\em Geometric methods in algebra and number theory}, volume 235
  of {\em Progr. Math.}, pages 251--282. Birkh\"auser Boston, Boston, MA, 2005.

\bibitem[SS16]{saito-schnell}
Morihiko Saito and Christian Schnell.
\newblock Fields of definition of {H}odge loci.
\newblock In {\em Recent advances in {H}odge theory}, volume 427 of {\em London
  Math. Soc. Lecture Note Ser.}, pages 275--291. Cambridge Univ. Press,
  Cambridge, 2016.

\bibitem[Urb22]{urbanik-absolute}
David Urbanik.
\newblock Absolute {H}odge and {$\ell$}-adic monodromy.
\newblock {\em Compos. Math.}, 158(3):568--584, 2022.

\bibitem[Urb24]{urbanik-image-model}
David Urbanik.
\newblock On the transcendence of period images.
\newblock {\em J. Differ. Geom.}, 127(2):615--662, 2024.

\bibitem[Voi07]{voisin-absolute}
Claire Voisin.
\newblock Hodge loci and absolute {Hodge} classes.
\newblock {\em Compos. Math.}, 143(4):945--958, 2007.

\bibitem[Zil02]{Zilber}
Boris Zilber.
\newblock Exponential sums equations and the {S}chanuel conjecture.
\newblock {\em J. London Math. Soc. (2)}, 65(1):27--44, 2002.

\end{thebibliography}

\end{document}